\documentclass[10pt,twoside]{siamart1116}
\usepackage[english]{babel}
\usepackage{graphicx,epstopdf,epsfig}
\usepackage{amsfonts,epsfig,fancyhdr,graphics, hyperref,amsmath,amssymb}
\usepackage{xcolor,soul,cancel}
\def\cvd{~\vbox{\hrule\hbox{%
     \vrule height1.3ex\hskip0.8ex\vrule}\hrule } }

\newcommand{\Names}{Mackenzie Cox, Weston M. Grewe, Grace K. Hochrein, Linda J. Patton, Ilya M. Spitkovsky}
\newcommand{\Title}{Non-parallel Flat Portions on the Boundaries \\ of Numerical Ranges of 4-by-4 Nilpotent Matrices}

\newtheorem{thm}[theorem]{Theorem}
\newtheorem{remark}[theorem]{Remark}
\newtheorem{example}[theorem]{Example}
\newtheorem{prop}[theorem]{Proposition}

\newtheorem{lem}[theorem]{Lemma}

\newcommand{\tr}{\operatorname{tr}}

  \newcommand{\re}{\operatorname{Re}}
  \newcommand{\im}{\mathrm{Im}}
 
\renewcommand{\theequation}{\thesection.\arabic{equation}}
\renewtheorem{theorem}{Theorem}[section]

\begin{document}

\setcounter{page}{1}

\thispagestyle{empty}

 \title{\Title\thanks{The first four authors were supported by the Bill and Linda Frost Fund. The fifth author was supported in part by Faculty Research funding from the Division of Science and Mathematics, NYUAD.}}

\author{
	Mackenzie Cox\thanks{Mathematics Department,
		California Polytechnic State University, San Luis Obispo, CA 93402, USA
		(readmac@gmail.com)}
	\and
	Weston M. Grewe\thanks{Mathematics Department,
		California Polytechnic State University, San Luis Obispo, CA 93402, USA
		(we\nolinebreak ston.grewe@ucdenver.edu)}
	\and
	Grace K. Hochrein\thanks{Mathematics Department,
		California Polytechnic State University, San Luis Obispo, CA 93402, USA
		(ghochrei@umich.edu)}
	\and
	Linda J. Patton\thanks{Mathematics Department,
		California Polytechnic State University, San Luis Obispo, CA 93402, USA
		(lpatton@calpoly.edu)}
	\and~Ilya~M.~Spitkovsky\thanks{Division of Science, New York  University Abu Dhabi (NYUAD), Saadiyat Island,
		P.O. Box 129188, Abu Dhabi, UAE (ims2@nyu.edu, ilya@math.wm.edu, imspitkovsky@gmail.com). }}

\markboth{\Names}{\Title}

\maketitle

\begin{abstract}
The 4-by-4 nilpotent matrices the numerical ranges of which have non-parallel flat portions on their boundary that are on lines equidistant from the origin are characterized. Their numerical ranges are always symmetric about a line through the origin and all possible angles between the lines containing the flat portions are attained. 
\end{abstract}




\section{Introduction}

Let $T$ be a bounded linear operator on a complex Hilbert space $H$, where the inner product of vectors $w$ and $v$ in $H$ is denoted $\langle w, v \rangle$ and the norm of $v$ is denoted $\| v \|.$ The numerical range of $T$ is the subset of the complex plane $\mathbb{C}$ defined by 

$$W(T)= \left\{ \langle T v, v \rangle \, \colon v \in H, \| v \|=1 \right\}.$$

The set $W(T)$ is bounded by the operator norm $\|T \|$ and the Toeplitz-Hausdorff theorem \cite{Toe18}, \cite{Hau} states that $W(T)$ is convex.

 If $H=\mathbb{C}^n$ with the standard inner product $\langle w, v \rangle=\sum_{j=1}^n w_j \overline{v}_j$,  then the numerical range is also closed. In this case the operator $T$ is represented by an $n$-by-$n$ matrix $A$. The set $W(A)$ contains the spectrum $\sigma(A)$ of $A$ and if $A$ is normal, then $W(A)$ is the convex hull of $\sigma(A)$. When $n=2$ and $A$ is not normal, the numerical range is an elliptical disk. These results and other basic facts about the numerical range are in many sources, including \cite[Chapter 1]{HJ2} and the recent monograph \cite[Chapter 6]{DGSV}.

 In 1951, Kippenhahn  \cite{Ki} (see also the English translation \cite{Ki08}) described the numerical range of a matrix $A$ in terms of a certain associated homogeneous polynomial.  The natural domain for these polynomials is the complex projective plane $\mathbb{P}_2(\mathbb{C})$. Define $\mathbb{P}_2(\mathbb{C})$ to be the set of equivalence classes $[(x,y,z)]$ of ordered nonzero triples $(x,y,z) \in \mathbb{C}^3$ such that $(x,y,z) \sim (x',y',z')$ if and only if there exists $\lambda \in \mathbb{C} \setminus \left\{0\right\}$ such that $(x,y,z)=\lambda(x',y',z')$. Any element $[(u,v,w)]$ of $\mathbb{P}_2(\mathbb{C})$ defines both a point and a line 

$$\ell_{[(u,v,w)]}=\left\{ [(x,y,z)] \in \mathbb{P}_2(\mathbb{C}) \, \vert \, ux+vy+wz=0 \right\}.$$

 The zero set of a homogenous polynomial $p$ of three variables in $\mathbb{P}_2(\mathbb{C})$ is called an algebraic curve. It can be defined in point coordinates, resulting in the curve
 
 \begin{equation*}\label{pointcoords}
\Gamma'=\left\{ [(x,y,z)] \in \mathbb{P}_2(\mathbb{C}) \, \vert \, p(x,y,z)=0 \right\},
\end{equation*}
or in line coordinates, resulting in the dual curve $\Gamma$ which is the envelope of the set of all lines

\begin{equation}\label{linecoords}
\left\{ \ell_{[(u,v,w)]} \, \vert \, p(u,v,w)=0 \right\}.
\end{equation}

 By duality, a line is tangent to $\Gamma$ if and only if it defines a point on $\Gamma'$ and vice-versa. The theory of algebraic curves (such as results which count curve intersections) can be stated most concisely in $\mathbb{P}_2(\mathbb{C})$, but by identifying the point $(u,v)$ with $(u,v,1)$, one can consider the set $\mathbb{C}^2$ and the complex plane (identified with $\mathbb{R}^2$) as subsets of $\mathbb{P}_2(\mathbb{C})$. In particular, the real affine part of  $\Gamma$ (or $\Gamma'$) is the set of all lines (respectively, points) determined by $(u,v) \in \mathbb{R}^2$ such that $p(u,v,1)=0$. See \cite{Shap2} or \cite{DGSV} for concise discussions of facts about algebraic curves and the projective plane that are useful for the study of numerical ranges.

If $A$ is an $n$-by-$n$ matrix, define the associated homogeneous polynomial $p_A$ by
\begin{equation}\label{pa} p_A(u,v,w)=\det(uH+vK+wI),
\end{equation}
 where $ H=\frac{A+A^*}{2}:=\re A, \quad K=\frac{A-A^*}{2i}:=\im A$, and $I$ is the $n$-by-$n$ identity matrix. Kippenhahn showed that the numerical range of $A$ is the convex hull of the real affine part of the curve associated with  $p_A$  in line coordinates. That is, if  $\mathcal{C}(A)$ is the real affine part of the curve associated with $p_A$ whose tangent lines are defined as in \eqref{linecoords}, then the numerical range $W(A)$ is the convex hull of $\mathcal{C}(A)$ in $\mathbb{C}$. Therefore, if $p_A(u_0, v_0, 1)=0$ for $(u_0,v_0) \in \mathbb{R}^2$, then the line $\ell_{(u_0,v_0)}$ defined by
 
 \begin{equation}\label{ell}
\ell_{(u_0,v_0)}=\left\{ (x,y) \in \mathbb{R}^2 \, \vert \, u_0 x+v_0 y+1=0 \right\}.
 \end{equation}
 is tangent to the curve $\mathcal{C}(A)$ and the collection of these tangent lines determines the curve, the convex hull of which is $W(A)$. Kippenhahn called $\mathcal{C}(A)$ the ``boundary generating curve" of $W(A)$. Following \cite{DGSV}, we adopt the term the {\em Kippenhahn curve} of $A$ and we'll refer to $p_A$ as the {\em numerical range (NR) generating polynomial} of $A$. We will use notation $\mathcal{C}'(A)$ to denote the real affine part of the point curve consisting of the points $(x,y) \in \mathbb{R}^2$ satisfying $p_A(x,y,1)=0$. 
 
 Kippenhahn used $p_A$ to classify the possible numerical ranges of 3-by-3 matrices as follows. If $A$ is not normal but is unitarily reducible to the direct sum of a 2-by-2 block and a 1-by-1 block, the numerical range of $A$ is the convex hull of an ellipse and a point, which could either be an elliptical disk if the point is in the disk or the convex hull of a point and an ellipse with flat portions on the boundary meeting at the point if it is exterior to the elliptical disk. If $A$ is unitarily irreducible, meaning it is not unitarily equivalent to a block diagonal matrix, then $W(A)$ could be either an elliptical disk, an ovular shape, or it could have one flat portion on the boundary. This flat portion arises if $p_A$ has a certain type of real singularity. In \cite{KRS}, \cite{RS05}, Kippenhahn's conditions were simplified to allow one to classify and thoroughly analyze the numerical range of a 3-by-3 matrix in terms of its entries. 
 
 A very large number of papers have been written about the numerical range since Kippenhahn's paper. However, progress on a complete analysis of the numerical range of $n$-by-$n$ matrices for specific $n>3$ has been gradual since complications arise quickly as the matrix size increases. 
 
An algebraic curve $\mathcal{C}$ (in line or point coordinates) is defined to be irreducible if $\mathcal{C}$ is the zero set of an irreducible polynomial. Determinant properties show that if the matrix $A$ is unitarily reducible, then the Kippenhahn curve is reducible, but the converse of this statement is false if $n>2$ \cite{Ki}, \cite{Ki08}. In \cite{ChiNa12}, Chien and Nakazato analyzed the Kippenhahn curve for $n=4$ and found 21 possible shapes for the numerical range corresponding to different combinations and types of singularities when the Kippenhahn curve is irreducible. 
   
One particularly tangible question to ask about the numerical range of a matrix $A$ is whether there are flat portions on the boundary and if so, how many? If a matrix $A$ is block diagonal and the numerical ranges of the blocks do not satisfy a restrictive subset condition, the convexity of $W(A)$ will produce line segments between the numerical ranges of the blocks as is the case for normal $A$ or, if $n=3$,  in one of the cases mentioned above. However, flat portions can also appear when the matrix is unitarily irreducible if there are singularities in the NR generating polynomial; this also already materializes in the $n=3$ case.  As described in \cite{GauWu081}, since an irreducible homogeneous polynomial of degree $n$ has at most $(n-1)(n-2)/2$ singularities, that value is an upper bound for the number of flat portions on the boundary of an $n$-by-$n$ matrix $A$ when $p_A$ is irreducible. If $p_A$ is reducible, there is a slightly higher upper bound on the number of flat portions given by $n(n-1)/2$. These upper bounds are not sharp in general. 
 
 Recall that a matrix $A$ is called nilpotent if there exists a positive integer $m$ such that $A^m=0$. Since the spectrum of a nilpotent matrix $A$ is the singleton $\left\{0\right\}$, the upper triangular form of $A$ has zeros on the main diagonal and $W(A)$ contains the origin. These conditions simplify the analysis of flat portions on the boundary of the numerical range. Note that a non-zero nilpotent 2-by-2 matrix has a circular disk as the numerical range and thus no flat portions on its boundary. As Kippenhahn showed, a 3-by-3 matrix has at most one flat portion on its numerical range boundary; examples with one flat portion where the matrix is nilpotent exist, and the respective criterion is in \cite[Theorem~4.1]{KRS}. Gau and Wu \cite{GauWu081} conjectured  that the boundary of the numerical range of an $n$-by-$n$ nilpotent matrix contains at most $n-2$ flat portions if $2 \leq n \leq 5$ and contains at most $2(n-4)$ flat portions if $n \geq 6$. They provided evidence for their conjecture by proving it for the special case of an $n$-by-$n$ nilpotent matrix $A$ which has an $(n-1)$ -by-$(n-1)$ principal submatrix the numerical range of which is a circular disk centered at the origin. In \cite{MPST}, Militzer, Tsai, and the last two authors of the current paper proved Gau and Wu's conjecture in the $n=4$ case. That is, the boundary of the numerical range of a 4-by-4 nilpotent matrix has at most two flat portions. The paper \cite{MPST} also included characterizations of numerical ranges of 4-by-4 nilpotent matrices with two parallel flat portions on their boundary. In particular, it was observed there that such flat portions have to be equidistant from the origin.
 
As it happens, two non-parallel flat portions can also materialize. One such example, provided in \cite{GauWu081}, is delivered by

 \begin{equation}\label{gw} 
 A=\left[\begin{array}{cccc}0 & 1 & 0 & -2 \\0 & 0 & 2 & $i$ \\0 & 0 & 0 & 1 \\0 & 0 & 0 & 0\end{array}\right].
 \end{equation}
The flat portions on $\partial W(A)$ are on lines equidistant from the origin and $W(A)$ is symmetric about the imaginary axis. In Section 3, we show that for any nilpotent 4-by-4 matrix with flat portions on the boundary of the numerical range that are on lines equidistant from the origin, the numerical range is symmetric about a line through the origin. Section 4 contains the main result of the paper, Theorem \ref{allangles}, where we characterize the 4-by-4 nilpotent matrices such that the boundary of the numerical range contains two nonparallel flat portions on lines equidistant from the origin. We show that every possible angle between those lines is attained and provide examples of families of matrices that exhibit all angles in Section 5. It seems to be an open question whether one can construct a 4-by-4 nilpotent matrix with flat portions on the boundary of the numerical range that are on non-parallel lines at different distances from the origin.

\section{Methods to verify existence of flat portions}
\setcounter{equation}{0}
For an $n$-by-$n$ matrix $A$, there are two well-known methods to determine whether or not the boundary of $W(A)$ contains a flat portion on a given line. One method involves singularities of the NR generating polynomial $p_A$ and the other involves the eigensystem of the real part of a suitably chosen rotation of $A$. 
Although equivalent formulations have been used in the literature, we will include precise statements of the required results and some proofs.

\begin{prop}\label{flattest} Let $A$ be an $n$-by-$n$ complex matrix with NR generating polynomial $p:=p_A$ defined by \eqref{pa}. Assume $0 \in W(A)$ and $(u_0,v_0) \in \mathbb{R}^2$. Let $a_{11}=p_{uu}(u_0,v_0,1),$ $a_{12}=p_{uv}(u_0,v_0,1)$, and $a_{22}=p_{vv}(u_0,v_0,1)$. There is a flat portion on the intersection of the boundary of $W(A)$ and the line $\ell_{(u_0,v_0)}$ defined in \eqref{ell} if 
\begin{equation*}\label{gradient}
\nabla p(u_0,v_0,1)=(0,0,0),
\end{equation*}
where $\nabla p$ denotes the gradient of $p$,
\begin{equation}\label{distinct}
 a_{22} \ne 0, \, \ a_{11}a_{22}  <a_{12}^2, \, \text{ and } a_{11}u_0^2+2a_{12} u_0 v_0 +a_{22} v_0^2 \ne 0, 
\end{equation}
and
\begin{equation}\label{extremeline}
p(u_0,v_0,\gamma)=0 \, \text{ implies } \gamma \leq 1.
\end{equation}
When this occurs, 
 \begin{align}\label{pointsoncurve}
\begin{split}
(s_1, t_1)&=\left(\frac{a_{12}-\sqrt{a_{12}^2-a_{11}a_{22}}}{-a_{22}v_0-(a_{12}-\sqrt{a_{12}^2-a_{11}a_{22}} )u_0},\frac{a_{22}}{-a_{22}v_0-(a_{12}-\sqrt{a_{12}^2-a_{11}a_{22}} )u_0} \right),\\
(s_2, t_2)&=\left(\frac{a_{12}+\sqrt{a_{12}^2-a_{11}a_{22}}}{-a_{22}v_0-(a_{12}+\sqrt{a_{12}^2-a_{11}a_{22}} )u_0},\frac{a_{22}}{-a_{22}v_0-(a_{12}+\sqrt{a_{12}^2-a_{11}a_{22}} )u_0} \right). 
\end{split}
\end{align} are the endpoints of the flat portion on the boundary of $W(A)$. 

\end{prop}

\begin{remark} The length $L$ of the flat portion is the distance between the points $(s_1,t_1)$ and $(s_2, t_2)$, which is
\begin{equation}\label{lengthofflat}
L=\frac{2 \sqrt{a_{12}^2-a_{11}a_{22}} \sqrt{u_0^2+v_0^2}}{|a_{11}u_0^2+2a_{12}u_0v_0+a_{22}v_0^2|}.
\end{equation}

\end{remark}

\begin{proof} Let $A$ be an $n$-by-$n$ complex matrix with NR generating polynomial $p$ that satisfies the hypotheses above. The condition $\nabla p(u_0,v_0,1)=(0,0,0)$ implies $p(u_0, v_0, 1)=0$ as well. Expand $p(u,v,1)$ in a Taylor series about the point $(u_0,v_0)$ to obtain

\begin{align*}p(u,v,1)&=0+ 0 (u-u_0)+0 (v-v_0)+ \frac{1}{2}\left(a_{11}(u-u_0)^2+ 2a_{12} (u-u_0) (v-v_0) + a_{22} (v-v_0)^2 \right) + \ldots  \\
&= \frac{1}{2}\left(a_{11}(u-u_0)^2+ 2 a_{12}(u-u_0) (v-v_0) + a_{22} (v-v_0)^2 \right) + \ldots \\
\end{align*}
Since by assumption not all second order partial derivatives are zero, we can approximate $p$ near $(u_0,v_0,1)$ with its second order terms and find the tangent line(s) with those terms. 

\begin{align*}p(u,v,1)&\approx \frac{a_{22}}{2}\left( v-v_0+\frac{u-u_0}{a_{22}}\left(a_{12}-\sqrt{a_{12}^2-a_{11}a_{22}} \right)\right) \left( v-v_0+\frac{u-u_0}{a_{22}}\left(a_{12}+\sqrt{a_{12}^2-a_{11}a_{22}} \right)\right)\\
&= \frac{a_{22}}{2}\left( \frac{a_{12}-\sqrt{a_{12}^2-a_{11}a_{22}} }{a_{22}} u+v-v_0-\frac{a_{12}-\sqrt{a_{12}^2-a_{11}a_{22}} }{a_{22}} u_0\right) \cdot \\
&\left( \frac{a_{12}+\sqrt{a_{12}^2-a_{11}a_{22}} }{a_{22}} u+v-v_0-\frac{a_{12}+\sqrt{a_{12}^2-a_{11}a_{22}} }{a_{22}} u_0\right) .
\end{align*}

Therefore there are two (possibly coinciding) tangent lines to  $\mathcal{C}'(A)$, the real affine curve in point coordinates, at $(u_0,v_0,1)$. Using \eqref{distinct}, they can be written in the standard form $s_j u + t_j v +1=0$ (for $j=1,2$) with  $(s_1, t_1)$ and $(s_2, t_2)$ given by \eqref{pointsoncurve}.  

Since $a_{12}^2-a_{11}a_{22}>0$, the points $(s_1, t_1)$ and $(s_2, t_2)$ are real and distinct. Thus the curve $\mathcal{C}'(A)$ has two distinct tangent lines, $\ell_{(s_1,t_1)}$ and $\ell_{(s_2,t_2)}$   at  the point $(u_0,v_0)$ which means the dual $\mathcal{C}(A)$ contains two distinct points $(s_1, t_1)$ and $(s_2, t_2)$ on the line $\ell_{(u_0,v_0)}$. Therefore the convex hull of $\mathcal{C}(A)$ includes the line segment between $(s_1, t_1)$ and $(s_2, t_2)$. We only need to show that $\ell_{(u_0,v_0)}$ is on the boundary of $W(A)$, as opposed to being a tangent line to a component of $\mathcal{C}(A)$ in the interior of $W(A)$,  to complete the proof. 

If $v_0 \ne 0$, the line $\ell_{(u_0,v_0)}$ has $y$-intercept $-1/v_0$. The set $W(A)$ will have $\ell_{(u_0,v_0)}$ as a support line if the parallel lines $\ell_{u_0/\gamma,v_0/\gamma}$ with $p(u_0/\gamma, v_0/\gamma, 1)=0$ are either all above or all below $\ell_{(u_0,v_0)}$. The line $\ell_{u_0/\gamma,v_0/\gamma}$ has $y$-intercept $-\gamma/v_0$. If $v_0<0$, then $\ell_{(u_0,v_0)}$ crosses the $y$-axis above the origin and since the origin is in $W(A)$, it suffices to show all other tangent lines are below $\ell_{(u_0,v_0)}$. If $v_0>0$, then $\ell_{(u_0,v_0)}$ crosses the $y$-axis below the origin and it suffices to show other tangent lines are all above $\ell_{(u_0,v_0)}$. The hypothesis \eqref{extremeline} is enough to prove both these cases. If $v_0=0$, a similar argument shows $\ell_{(u_0,0)}$ is either the leftmost or rightmost vertical line tangent to the Kippenhahn curve. 

Therefore the segment between $(s_1, t_1)$ and $(s_2, t_2)$ on the line $\ell_{(u_0,v_0)}$ is a flat portion on the boundary of $W(A)$. \end{proof} 

Since the focus of this paper is the case where $A$ is a 4-by-4 nilpotent matrix, we include a result from \cite{MPST} about the special form of the NR generating polynomial $p_A$ for such a matrix. Let $$\beta_0=\tr(A^*AA^*A); \quad \beta_{11}=\tr(A A^*); \quad \beta_{22} =\tr(A^2 (A^*)^2) .$$
 In addition, we set $\beta_{21}= \tr(A^2 A^*)$ and $\beta_{31} = \tr(A^3 A^*)$; trace properties then imply that $\overline{\beta}_{21}=\tr(\left[A^* \right]^2 A)$ and $\overline{\beta}_{31}=\tr(\left[A^* \right]^3 A)$.

\begin{lem}\label{nilpotentbgc}{\em\cite{MPST}} Let $A$ be a 4-by-4 nilpotent matrix. The NR generating polynomial associated with $A$ is defined by
\begin{equation}\label{nilbgc}\begin{split} p_A(u,v,w)&=c_1 u^4+c_2 u^3 v+c_3 u^3w+ (c_1+c_4) u^2 v^2 +c_5 u^2 w^2\\+&c_6 u^2vw+c_2 uv^3+c_3uv^2w+c_4v^4+c_6v^3w+c_5 v^2w^2+w^4,\end{split}\end{equation} where the coefficients $c_j$ are given below.

$$c_1=-\frac{1}{16} \left(2 \re \beta_{31} + \beta_{22} +\frac{1}{2}\beta_0-\frac{1}{2} \beta_{11}^2 \right),$$
$$c_2=-\frac{1}{4}\im \beta_{31}, \quad c_3=\frac{1}{4} \re\beta_{21}, $$
$$c_4 = \frac{1}{16} \left(2 \re \beta_{31} - \beta_{22} -\frac{1}{2}\beta_0+\frac{1}{2} \beta_{11}^2 \right),\quad c_5 = -\frac{1}{4} \beta_{11},  \quad c_6= \frac{1}{4} \im \beta_{21}.$$

\end{lem}

 Note that the polynomial \eqref{nilbgc} has a singularity at $(u,v,w)$ if and only if $\nabla p_A(u,v,w)=(0,0,0)$, which occurs when
 
 \begin{align}\label{singularities}
\begin{split}
(4u^3+2uv^2) c_1+(3u^2v+v^3)c_2+(3u^2w+v^2w)c_3+2uv^2 c_4+ 2uw^2c_5+2uvwc_6 &=0, \\
2u^2vc_1+(u^3+3uv^2)c_2+2uvwc_3+(2u^2v+4v^3)c_4+2vw^2c_5+(u^2w+3v^2w)c_6&=0,\\
(u^3+uv^2)c_3+(2u^2w+2v^2w)c_5+(u^2v+v^3)c_6 &=-4 w^3. 
\end{split}
\end{align}
Now we discuss using the eigensystem to analyze flat portions.  Given a square matrix $M$, the maximum real part of an element in $W(M)$ is the maximum eigenvalue $d$ of $\re M$. Therefore $x=d$ is the rightmost vertical support line of $W(A)$, and $\langle M \xi, \xi \rangle$ is on the boundary of $W(M)$ and also on $x=d$, if and only if $\xi$ is a unit eigenvector of $\re M$ corresponding to the eigenvalue $d$. 
	
Consequently, if there are two (linearly independent) unit eigenvectors $\xi_1$ and $\xi_2$ of $\re M$ corresponding to $d$ with $\im \langle M \xi_1, \xi_1 \rangle \ne \im \langle M \xi_2, \xi_2 \rangle$, then the entire line segment from $ \langle M \xi_1, \xi_1 \rangle$ to $ \langle M \xi_2, \xi_2 \rangle$ will be on the boundary of $W(M)$ and on the line $x=d$. 

In order to generalize this analysis to non-vertical flat portions,  one can rotate $W(M)$ through an angle $-\phi$ about the origin for each $\phi \in [0, 2 \pi)$ and apply the previous argument to $e ^{-i \phi} M$ in order to obtain the following criterion:

\begin{lem} \label{flatwithevec} Let $M$ be an n-by-n matrix. The boundary of  $W(M)$ contains a segment of the line $x \cos\phi + y \sin\phi=d$ if and only if $d$ is the largest eigenvalue of $\re (e^{-i \phi} M)$ and there are corresponding unit eigenvectors $\xi_1$ and $\xi_2$ that satisfy
\begin{equation*}
\im \langle e^{-i \phi} M \xi_1, \xi_1 \rangle \ne \im \langle e^{-i \phi} M \xi_2, \xi_2 \rangle.
\end{equation*}
\end{lem}

Some version of this result has been used in most papers describing flat portions on the boundary of the numerical range; for example, see \cite[Proposition~3.2]{KRS}, proof of Theorem~10 in \cite{BS041}, or \cite[Lemma~1.4]{GauWu08}.

Further analysis characterizing linearly independent eigenvectors which actually generate flat portions appeared in \cite{LLS13}. In \cite{MPST}, conditions on the entries of an upper triangular 4-by-4 nilpotent matrix which are necessary and sufficient for a flat portion to materialize on the boundary of the numerical range were derived. 

Note that if $d$ is the largest eigenvalue of $\re (e^{-i \phi} M)$ then $d I-\re (e^{-i \phi} M)$ is positive semidefinite and hence all the principal minors of $d I-\re (e^{-i \phi} M)$ are non-negative. Furthermore, if $\re (e^{-i \phi} M)$ has at least two linearly independent eigenvectors corresponding to the eigenvalue $d$ then $\det \left(d I-\re (e^{-i \phi }M)\right) =0$ and all principal 3 by 3 minors are also zero. The latter fact will be used in the proof of our main result, Theorem \ref{allangles}.

\section{Reflection symmetry of $W(A)$}\label{symmetry}
\setcounter{equation}{0}
The distance from the origin to the line $\ell_{(u,v)}$ defined in \eqref{ell} is $1/\sqrt{u^2+v^2}$. 
The point on $\ell_{(u,v)}$  that attains this minimum distance is $(x,y)= \left( \frac{-u}{u^2+v^2}, \frac{-v}{u^2+v^2} \right).$ When viewed as points in the complex plane, the involution $f(z)=-1/z$ maps a singularity $u+iv$ of $p_A$ to the distance-minimizing point $\frac{-u}{u^2+v^2}+i\frac{-v}{u^2+v^2}$ and vice-versa. Therefore, a pair of singularities is symmetric about a line through the origin if and only if the distance-minimizing points are symmetric about that same line. 

In particular, if $(u,v)=(r \cos\theta, r \sin\theta)$ then the point on $\ell_{(u,v)}$  that is closest to the origin is $$-\frac{1}{r} (\cos\theta, \sin\theta)=\frac{1}{r}( \cos(\theta+\pi), \sin(\theta+\pi));$$ the argument of this point is $\pi$ radians from the argument of the singularity $(u,v)$. If $\theta \in (0, \pi)$, then the slope of $\ell_{(u,v)}$ is $m=-\cot\theta$. 

In addition, any unitarily reducible 4-by-4 nilpotent matrix $A$ has at most one flat portion on the boundary of its numerical range \cite{MPST} so if $W(A)$ has two flat portions, then $A$ is irreducible and hence its eigenvalue 0 is in the interior of $W(A)$ \cite{Ki}, \cite{Ki08}. 

\begin{thm}\label{linesymmetry} Assume $A$ is a 4-by-4 complex nilpotent matrix such that $W(A)$ has two flat portions on its boundary. If the flat portions are on lines that are the same distance $d=1/r$ from the origin, then $W(A)$ is symmetric about a line through the origin. Namely, if the flat portions are on lines $\ell_{(u_1,v_1)}$ and $\ell_{(u_2,v_2)}$ where $(u_j,v_j)=(r\cos\theta_j, r\sin\theta_j)$ for $j=1,2$, then $W(A)$ is symmetric about the line 
\begin{equation}\label{lineofsymmetry}
 \sin\left( \frac{\theta_1+\theta_2}{2} \right) x-\cos \left( \frac{\theta_1+\theta_2}{2} \right) y=0. 
 \end{equation}
\end{thm}

\begin{proof}
Assume the hypotheses hold, so $W(A)$ has two flat portions on lines that are equidistant from the origin. 

If the lines are parallel, then Proposition 4.1 of \cite{MPST}
 and the remark following it show that $A$ is unitarily equivalent to a scalar multiple of a real matrix. The flat portions of the numerical range of this real matrix are horizontal and since the numerical range of a real matrix is symmetric about the real axis, they are parallel to the line of symmetry.  In general, the scaled matrix $A$ will therefore have flat portions on lines parallel to a line of symmetry. Without loss of generality the arguments of the singularities will satisfy $\theta_2=\theta_1+\pi$ so \eqref{lineofsymmetry} is equivalent to 
 $$  x \cos\theta_1+y \sin\theta_1=0,$$ which is parallel to the lines $\ell_{(u_1,v_1)}$ and $\ell_{(u_2,v_2)}$ that are described in the theorem statement. Thus the theorem holds when the flat portions are on parallel lines. 

Now assume the lines $\ell_{(u_1,v_1)}$ and $\ell_{(u_2,v_2)}$ containing the flat portions are not parallel. They will intersect at a point $s=|s| e^{i \tau}$. Denote the points minimizing distance from $\ell_{(u_1,v_1)}$ and $\ell_{(u_2,v_2)}$ to the origin by $P_1$ and $P_2$, respectively, where $P_j=-\frac{1}{r} (\cos\theta_j, \sin\theta_j)$ for $j=1,2$. We can assume without loss of generality that $\theta_2-\theta_1 \in (0, \pi)$. Draw line segments from the origin to $P_1$ and to $P_2$. These segments, denoted $\overline{P}_1$ and $\overline{P}_2$, are each of length $1/r$ and meet $\ell_{(u_1,v_1)}$ and $\ell_{(u_2,v_2)}$, respectively, in right angles. The line segment from the origin to $s$ is a common hypotenuse for two right triangles, each with base of length $1/r$, which must be congruent. See Figure \ref{sympic}.  The line segment from the origin to $s$ bisects the angle between $\overline{P}_1$ and $\overline{P}_2$. Therefore $\theta_2+\pi-\tau=\tau-(\theta_1+\pi)$, which shows $\tau=\frac{1}{2}(\theta_1+\theta_2)$ mod $2 \pi$. Now rotate $W(A)$ counterclockwise about the origin by $\pi-\tau$, or equivalently replace $A$ with $e^{i(\pi-\tau)} A$. This will place the point $s$ on the negative real axis. We will show the rotated numerical range is symmetric about the real axis, which will prove the original numerical range is symmetric about the line through the origin obtained by rotating the real axis counterclockwise by the angle $\tau-\pi$. This line satisfies the equation $\cos \left(\tau- \pi \right) y -\sin \left(\tau- \pi \right) x=0$, which reduces to \eqref{lineofsymmetry}.

Therefore for the remainder of this proof, assume the lines containing the flat portions intersect at a point on the negative real axis. 
The rotated pair of congruent triangles are symmetric about the real axis. This means the angle from the positive real axis to $\overline{P}_1$ is the negative of the angle from the positive real axis to $\overline{P}_2$ and the points $P_1$ and $P_2$ are symmetric about the real axis. By the involution argument in the first paragraph of this section, this means the singularities $(u_1,v_1)$ and $(u_2,v_2)$ that generated the lines are also symmetric about the real axis. Thus the form of the singularities is $(u,v)$ and $(u,-v)$ where $uv \ne 0$. 

\begin{figure}
\begin{center}
\includegraphics[scale=.6]{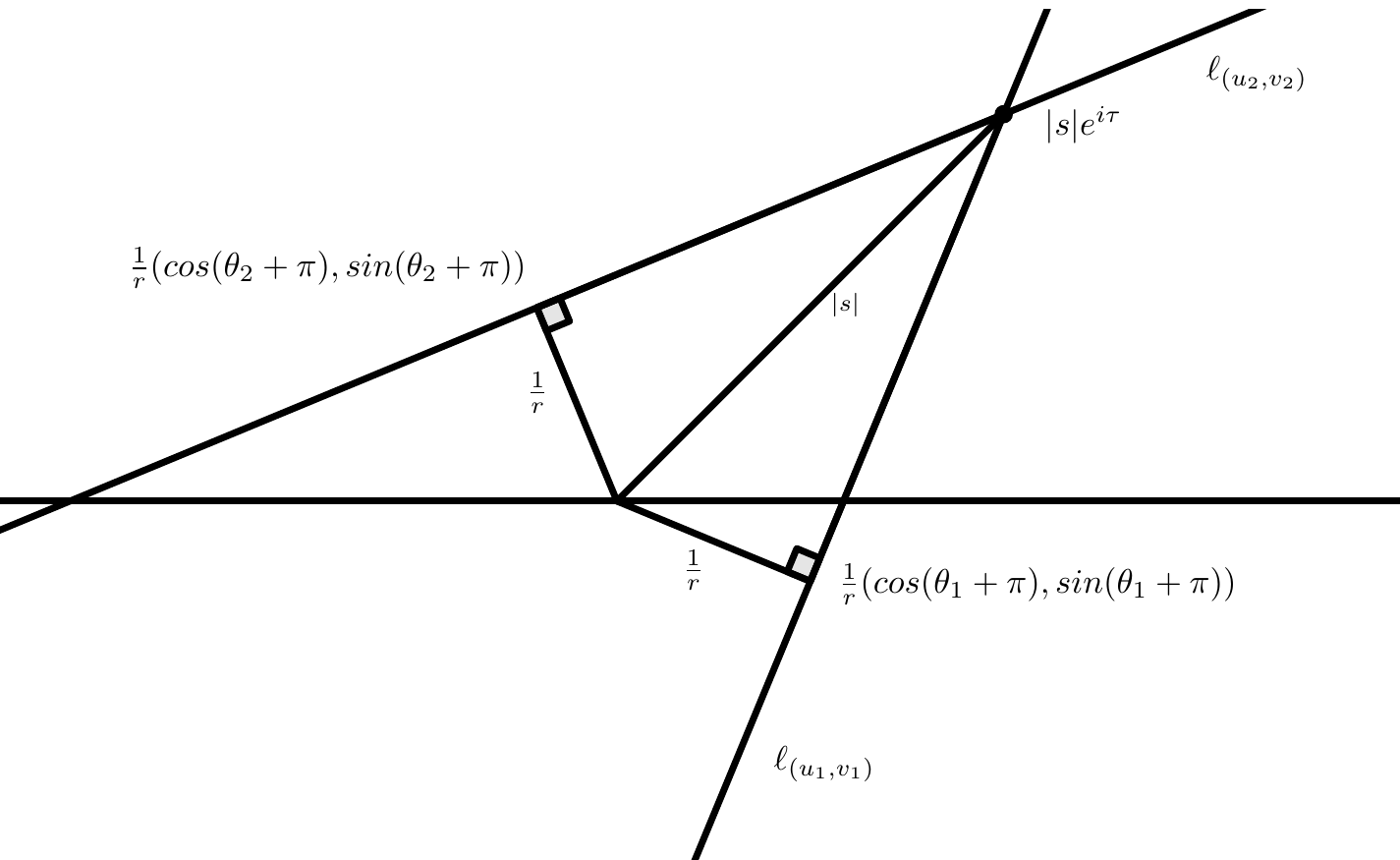}
\caption{}
\label{sympic}
\end{center}
\end{figure}

If the polynomial \eqref{nilbgc} has a singularity at $(u,v)$, then the system \eqref{singularities} with the augmented matrix

\begin{equation*}\label{system}
\begin{bmatrix}
 4 u^3+2 v^2 u & v^3+3 u^2 v & 3 u^2+v^2 & 2 u v^2 & 2 u & 2 u v & 0 \\
 2 u^2 v & u^3+3 v^2 u & 2 u v & 4 v^3+2 u^2 v & 2 v & u^2+3 v^2 & 0 \\
 0 & 0 & u^3+v^2 u & 0 & 2 u^2+2 v^2 & v^3+u^2 v & -4 \\
\end{bmatrix}
\end{equation*}
must be consistent.

If \eqref{nilbgc} has singularities at both $(u,v)$ and $(u,-v)$ then the system described by the augmented matrix

\begin{equation*}
\begin{bmatrix}
 4 u^3+2 v^2 u & v^3+3 u^2 v & 3 u^2+v^2 & 2 u v^2 & 2 u & 2 u v & 0 \\
 2 u^2 v & u^3+3 v^2 u & 2 u v & 4 v^3+2 u^2 v & 2 v & u^2+3 v^2 & 0 \\
 0 & 0 & u^3+v^2 u & 0 & 2 u^2+2 v^2 & v^3+u^2 v & -4 \\
 4 u^3+2 v^2 u & -v^3-3 u^2 v & 3 u^2+v^2 & 2 u v^2 & 2 u & -2 u v & 0 \\
 -2 u^2 v & u^3+3 v^2 u & -2 u v & -4 v^3-2 u^2 v & -2 v & u^2+3 v^2 & 0 \\
 0 & 0 & u^3+v^2 u & 0 & 2 u^2+2 v^2 & -v^3-u^2 v & -4 \\
\end{bmatrix}
\end{equation*}

must be consistent. 

Straightforward calculations show the system above is row equivalent to

\begin{equation}\label{uvuminusv}
\begin{bmatrix}
 4 u^3+2 v^2 u & v^3+3 u^2 v & 3 u^2+v^2 & 2 u v^2 & 2 u & 2 u v & 0 \\
 2 u^2 v & u^3+3 v^2 u & 2 u v & 4 v^3+2 u^2 v & 2 v & u^2+3 v^2 & 0 \\
 0 & 0 & u^3+v^2 u & 0 & 2 u^2+2 v^2 & v^3+u^2 v & -4 \\
 0 & -2 \left(v^3+3 u^2 v\right) & 0 & 0 & 0 & -4 u v & 0 \\
 0 & 2 u \left(u^2+3 v^2\right) & 0 & 0 & 0 & 2 \left(u^2+3 v^2\right) & 0 \\
 0 & 0 & 0 & 0 & 0 & -2 v \left(u^2+v^2\right) & 0 \\
\end{bmatrix}.
\end{equation}

Therefore, the coefficients $c_6=c_2=0$. This implies that the NR generating polynomial \eqref{nilbgc} is an even function of $v$. Hence the curve $\mathcal{C}'(A)$ and its dual, the Kippenhahn curve $\mathcal{C}(A)$, are symmetric about the real axis and thus the numerical range $W(A)$ is  also symmetric about the real axis as required.  \end{proof} 

\section{Main result}\label{chareqdis}
\setcounter{equation}{0}
In this section we prove the main result of the paper, which describes nilpotent matrices having numerical ranges with flat portions on a given pair of nonparallel lines that are equidistant from the origin. 

In the remainder of the paper we will use the abbreviations
	\begin{equation}\label{SC} S:=\sin(\theta/2), \quad C:=\cos(\theta/2), 
	\end{equation} 
suppressing the dependence on the argument $\theta/2$ for brevity.

\begin{thm}\label{allangles} Let $\theta \in (0, \pi)$ and let $d >0$. Assume $A$ is a 4-by-4 nilpotent matrix. Then the boundary of $W(A)$ contains two flat portions on non-parallel lines that are a distance $d$ from the origin and that meet at an angle $\theta$ if and only if $A$ is unitarily equivalent to a matrix of the form
\begin{equation}\label{Aform} e^{it} \left[\begin{array}{cccc}0 & x & \delta_1 & y \\0 & 0 & y & \delta_2 \\0 & 0 & 0 & x \\0 & 0 & 0 & 0\end{array}\right],
\end{equation}

where $t \in [0,2\pi)$,  $x \in (0, 2d)$, 

\begin{equation}\label{yrange} 0<y \leq \sqrt{\frac{16d^4-4d^2x^2}{4d^2-x^2 S^2}},
\end{equation}
and 

$$\delta_{1,2}= \frac{1}{2d} \left( xy S \pm \sqrt{x^2y^2 S^2+16d^4-4d^2(x^2+y^2)} \right).$$
Furthermore, in the case where $A$ satisfies these conditions, the length of each flat portion on the boundary is

$$L=\frac{8d^3xy C}{16d^4-x^2y^2S^2}.$$ \end{thm}

\begin{remark}\label{maxlength} For a fixed $d$ and $\theta$, the maximum length  occurs when $x=y$ equals the right side of  \eqref{yrange} above. To see this, note that $L$ is an increasing function of both $x$ and $y$. Thus $L$ is maximized on the boundary of the region defined by $0 < x \leq 2d$ and \eqref{yrange}. Symmetry of the expressions for $L$ and the upper limit of \eqref{yrange} produce a maximum value of $L$ when $x^2=y^2=4d^2/(1+C)$. In this case $L_{\mathrm{max}}=d.$

\end{remark}

\begin{remark} If two matrices have the form \eqref{Aform} for the same $d$ and $\theta$, then the numerical range only depends on $xy$. However, different $x$ values give non-unitarily similar matrices, because traces of words in $A$ and $A^*$ are unitary invariants, and

$$ \tr (A^2 (A^*)^2)=x^2y^2 \left( 2+\frac{x^2 S^2}{d^2} \right).$$

So in order for $A$ and $B$ to be unitarily similar, they must have the same value of $x$ and thus also the same value of $y$. This gives uncountably many different unitary equivalence classes with the same numerical range. Note that swapping the roles of $\delta_1$ and $\delta_2$ also produce non-unitarily equivalent matrices with the same numerical range unless $y$ takes on its maximum value and thus $\delta_1=\delta_2$. This contrasts with the 3-by-3 case result in \cite{KRS}, where all matrices with the same numerical range with one flat portion are unitarily equivalent. Moreover, in the 3-by-3 case the numerical range (and thus the unitary equivalence class of matrices) is completely determined by the endpoints of the flat portion.\end{remark}

\begin{remark}  Since $x,y$ in \eqref{Aform} are non-zero,  \cite[Theorem~4.1]{KRS} implies that
this matrix has a $3$-by-$3$ principal submatrix with a circular disk as its numerical range if and only if one of $\delta_j$ is equal to zero. So, matrices \eqref{Aform} with $\delta_1\delta_2\neq 0$ deliver an example when  conditions of \cite[Theorem~3.4]{GauWu081} are not satisfied while the conclusion still holds.
\end{remark}

\begin{remark} The matrix $A$ in \eqref{gw} is $-iB$ where 
$$B=\left[
\begin{array}{cccc}
 0 & i & 0 & -2 i \\
 0 & 0 & 2 i & -1 \\
 0 & 0 & 0 & i \\
 0 & 0 & 0 & 0 \\
\end{array}
\right].$$
The matrix $B$ is unitarily equivalent, via a diagonal unitary similarity, to \eqref{Aform} with $x=1$, $y=2$, $d=\frac{\sqrt{5}}{2}$, and $\theta=2 \arcsin{\frac{\sqrt{5}}{4}}$. In this case $\delta_1$ and $\delta_2$ are $0$ and $1$, respectively. 
\end{remark}

{\em Proof of Theorem~\ref{allangles}. } Necessity. Fix $d>0$ and $\theta \in (0,\pi)$ and assume $A$ is a 4-by-4 nilpotent matrix such that $W(A)$ has flat portions as described. By Theorem \ref{linesymmetry}, $W(A)$ is symmetric about a line through the origin. Without loss of generality, we may rotate $W(A)$ through an angle $-t$ so that the line of symmetry is the real axis and the non-parallel lines intersect at a point on the negative real axis. Therefore, we will assume the latter setting and show $A$ is unitarily equivalent to a matrix of the form \eqref{Aform} with $t=0$. 

Let $\ell_1$ and $\ell_2$ denote the lines containing the flat portions. Since these lines intersect on the negative real axis at a point $(-q,0)$ where $q>0$, we may assume $\ell_1$ intersects the imaginary axis at $(0,y_1)$ with $y_1>0$ while by symmetry $\ell_2$ intersects the imaginary axis at $(0,-y_1)$. Let $P_1$ and $P_2$ be the points on $\ell_1$ and $\ell_2$, respectively, that attain the minimum distance $d$ from the origin. Clearly the right triangle with vertices $(0,0)$, $(0,y_1)$ and $(-q,0)$ will contain $P_1$ on its hypotenuse with $d < \min\left\{ y_1, q \right\}$.  

By symmetry we will now restrict our analysis to the line $\ell_1$. As we have seen, if the singularity that corresponds to the flat portion on $\ell_1$ is at $(r \cos\alpha, r \sin\alpha)$, then $d=\frac{1}{r}$,  the slope of this line is $-\cot \alpha$ and the point $P_1$ on the line that achieves the minimum distance has coordinates $d(\cos (\alpha+\pi), \sin(\alpha +\pi))$. This means $\alpha+ \pi$ is an angle in quadrant two; hence $\alpha$ is in quadrant four. 

In addition to the results that follow from the singularity, note by hypothesis the line $\ell_1$ intersects the negative real axis at an angle $\theta/2$. Therefore the slope of this line is $\tan(\frac{\theta}{2})$. Setting $-\cot \alpha=\tan(\frac{\theta}{2})$ leads to $\alpha=\frac{\theta}{2}+\frac{\pi}{2} +k \pi$ for an integer $k$.  Since we know $\alpha$ is in quadrant four, it follows that $\alpha=\frac{\theta}{2} +\frac{3 \pi}{2}.$ Thus if $W(A)$ has the flat portions described in the preceding discussion, the singularities of the associated algebraic curve are at $(u_1, v_1)=\frac{1}{d} (\cos (\frac{\theta}{2} +\frac{3\pi}{2}), \sin (\frac{\theta}{2} +\frac{3 \pi}{2}))=\frac{1}{d}(S,-C)$ and $(u_2, v_2)=\frac{1}{d} (\cos (\frac{\theta}{2} +\frac{3 \pi}{2}), -\sin (\frac{\theta}{2} +\frac{3\pi}{2}))=\frac{1}{d}(S,C)$. 

The system of equations that the coefficients of the NR generating polynomial associated with $A$ must satisfy to have singularities at $(u,v)$  and $(u,-v)$ when $uv \ne 0$ (which holds here) is equivalent to \eqref{uvuminusv}. If we completely row reduce that system, we obtain

\begin{equation*}
\left[
\begin{array}{ccccccc}
 1 & 0 & 0 & 0 & -\frac{1}{u^2} & 0 & \frac{3 u^2+2 v^2}{u^2 \left(u^2+v^2\right)^2} \\
 0 & 1 & 0 & 0 & 0 & 0 & 0 \\
 0 & 0 & 1 & 0 & \frac{2}{u} & 0 & -\frac{4}{u^3+u v^2} \\
 0 & 0 & 0 & 1 & 0 & 0 & \frac{1}{\left(u^2+v^2\right)^2} \\
 0 & 0 & 0 & 0 & 0 & 1 & 0 \\
 0 & 0 & 0 & 0 & 0 & 0 & 0 \\
\end{array}
\right].
\end{equation*}

Now set $(u,v)=\frac{1}{d}(S,C)$ to obtain the following system of equations that the coefficients of the NR generating polynomial must satisfy to have two flat portions on lines a distance $d$ from the origin that meet at an angle $\theta$ on the negative real axis. 

$$R_0=\left[
\begin{array}{ccccccc}
 1 & 0 & 0 & 0 &-\frac{d^2}{S^2} & 0 & d^4+ \frac{2d^4}{S^2} \\
 0 & 1 & 0 & 0 & 0 & 0 & 0 \\
 0 & 0 & 1 & 0 & \frac{2 d}{S} & 0 & -4 \frac{d^3}{S}\\
 0 & 0 & 0 & 1 & 0 & 0 & d^4 \\
 0 & 0 & 0 & 0 & 0 & 1 & 0 \\
 0 & 0 & 0 & 0 & 0 & 0 & 0 \\
\end{array}
\right].$$

Hence we know the coefficients of the curve must satisfy:

$$c_1=\frac{d^2}{S^2} c_5+2 \frac{d^4}{S^2}+d^4,$$

$$c_2=c_6=0,$$

$$c_4=d^4,$$
and

$$c_3=-2 \frac{d}{S} c_5-4 \frac{d^3}{S} .$$

A unitary similarity argument shows that we may assume our matrix $A$ has the form

$$A=\left[\begin{array}{cccc}0 & x & a_1+ib_1 & a_2+ib_2 \\0 & 0 & y & a_3+ib_3 \\0 & 0 & 0 & z \\0 & 0 & 0 & 0\end{array}\right],$$
where $x,y,z$ are non-negative and $a_j,b_j$ are real for $j=1,2,3.$ 

Now we can also compute the NR generating polynomial coefficients from Lemma \ref{nilpotentbgc} to obtain

$$c_1=\frac{1}{16} \left(a_1^2 \left(a_3^2+b_3^2\right)-2 a_1 (a_2 a_3 y+a_3 x z+b_2 b_3 y)+a_2^2 y^2-2 x z a_2 y-2xzb_1 b_3\right.$$
$$\left.+2 a_2 b_1 b_3 y+a_3^2 b_1^2-2 a_3 b_1 b_2 y+b_1^2 b_3^2+b_2^2 y^2+x^2 z^2\right),$$

$$c_2=\frac{1}{4} b_2 x y z,$$

$$c_3=\frac{1}{4} (a_1 a_2 z+a_1 x y+a_2 a_3 x+a_3 y z+b_1 b_2 z+b_2 b_3 x),$$

$$c_4=\frac{1}{16} \left(a_1^2 \left(a_3^2+b_3^2\right)-2 a_1 (a_2 a_3 y+a_3 x z+b_2 b_3 y)\right.$$
$$\left.+a_2^2 y^2+2 a_2 b_1 b_3 y+2 a_2 x y z+a_3^2 b_1^2-2 a_3 b_1 b_2 y+b_1^2 b_3^2-2 b_1 b_3 x z+b_2^2 y^2+x^2 z^2\right),$$

$$c_5=\frac{1}{4} \left(-a_1^2-a_2^2-a_3^2-b_1^2-b_2^2-b_3^2-x^2-y^2-z^2\right),$$
and

$$c_6=\frac{1}{4} (-a_1 b_2 z+a_2 b_1 z+a_2 b_3 x-a_3 b_2 x-b_1 x y-b_3 y z).$$

Note $c_1-c_4=-\frac{1}{4} \mathrm{Re} \beta_{31}=-\frac{1}{4} a_2 x y z.$ Also, 

$$c_1-c_4=\frac{d^2}{S^2} c_5+2 \frac{d^4}{S^2}=-\frac{d}{2S} c_3.$$

The identities $c_2=c_6=0$ hold because $W(A)$ is symmetric about the real axis. If $c_3=0$ then $W(A)$ is also symmetric about the imaginary axis. This would result in four flat portions on the boundary, which is impossible for a 4-by-4 nilpotent matrix $A$. Therefore $c_1=c_4$ is impossible. So the variables $x$, $y$, $z$, and $a_2$ are all nonzero. 

Since $c_2=0$, this implies $b_2=0$. 

Now consider the triangle in the third quadrant with a vertex at the origin, a vertex at the point  $(-q,0)$ where the lines containing the flat portions meet, and a vertex on the flat portion at the point $P_2$ that attains the minimum distance to the origin. Since this is a right triangle and the angle where the support line meets the negative real axis is $\theta/2$, it follows that the angle between the line segment from the origin to $P_2$ and the negative real axis is $\frac{\pi-\theta}{2}$ By rotating clockwise, we see that $W(e^{i (\frac{\pi-\theta}{2})} A)$ has $x=-d$ as a support line. Set $M=e^{i (\frac{\pi-\theta}{2})} A$ and $\phi=\pi$ in Lemma \ref{flatwithevec} to see that $d$ is the maximum eigenvalue of the Hermitian matrix $-\re M=-\frac{1}{2}(e^{i \left(\frac{\pi-\theta}{2}\right)}A+e^{-i \left(\frac{\pi-\theta}{2}\right)} A^*)$ and there are at least two corresponding linearly independent eigenvectors. Let $ H=d I+(1/2)  (e^{i \left(\frac{\pi-\theta}{2}\right)}A+ e^{-i \left(\frac{\pi-\theta}{2}\right)} A^*)$. Then 

$$H=\left[
\begin{array}{cccc}
 d & \frac{1}{2} e^{i\frac{\pi -\theta }{2}} x & \frac{1}{2} (a_1+b_1 i) e^{i\frac{\pi -\theta }{2}}& \frac{1}{2} a_2 e^{i\frac{\pi -\theta }{2}} \\
 \frac{1}{2} e^{i\frac{\theta -\pi }{2}} x & d & \frac{1}{2} e^{i\frac{\pi -\theta }{2}}y & \frac{1}{2} (a_3+b_3 i) e^{i\frac{\pi -\theta }{2}} \\
 \frac{1}{2}(a_1-ib_1) e^{i\frac{\theta -\pi }{2}} & \frac{1}{2}e^{i\frac{\theta -\pi }{2}}y & d & \frac{1}{2} e^{i\frac{\pi -\theta }{2}}z \\
 \frac{1}{2} a_2  e^{i\frac{\theta -\pi }{2}} & \frac{1}{2}(a_3-ib_3) e^{i\frac{\theta -\pi }{2}}& \frac{1}{2} e^{i\frac{\theta -\pi }{2}} z & d \\
\end{array}
\right]$$
has determinant zero and all principal 3-by-3 minors are also zero. Since we obtain the same support line by rotating in the opposite direction, the above is also true for $\hat{H}=d I+(1/2)  (e^{-i \left(\frac{\pi-\theta}{2}\right)}A+ e^{i \left(\frac{\pi-\theta}{2}\right)} A^*)$.

Therefore we have:

$$\frac{1}{4} \left(-d \left(a_1^2+b_1^2-4 d^2+x^2+y^2\right)+a_1 x y S+b_1 x y C\right)=0,$$
$$\frac{1}{4} \left(-d \left(a_1^2+b_1^2-4 d^2+x^2+y^2\right)+a_1 x y S-b_1 x y C\right)=0,$$
$$\frac{1}{4} \left(-d \left(a_2^2+a_3^2+b_3^2-4 d^2+x^2\right)+xa_2 a_3 S +xa_2 b_3 C \right)=0,$$
$$ \frac{1}{4} \left(-d \left(a_2^2+a_3^2+b_3^2-4 d^2+x^2\right)+x a_2 a_3S -x a_2 b_3C \right)=0,$$
$$ \frac{1}{4} \left(-d \left(a_1^2+a_2^2+b_1^2-4 d^2+z^2\right)+z a_1 a_2S +z a_2 b_1C\right)=0,$$
$$ \frac{1}{4} \left(-d \left(a_1^2+a_2^2+b_1^2-4 d^2+z^2\right)+za_1 a_2 S -z a_2 b_1C \right)=0,$$

$$ \frac{1}{4} \left(-d \left(a_3^2+b_3^2-4 d^2+y^2+z^2\right)+a_3 y z S+b_3 y z C\right)=0,$$

$$\frac{1}{4} \left(-d \left(a_3^2+b_3^2-4 d^2+y^2+z^2\right)+a_3 y z S-b_3 y z C\right)=0.$$

The value $ C \ne 0$ since $ \theta \in (0, \pi)$. We previously saw that $xyz \ne 0$. Therefore subtracting the second equation from the first yields
$b_1=0$. Subtracting the eighth equation from the seventh yields $b_3=0$. We now know our matrix $A$ is real with $x>0$, $y>0$, $z>0$ and $a_2 \ne 0$. In addition, the principal 2-by-2 determinants of $H$ are all non-negative so \begin{equation}\label{2d} 0< x,y,z \leq 2d\end{equation} and $|a_j| \leq 2d$ for $j=1,2,3.$ We are left with four distinct equations from the list above:

\begin{equation}\label{1of6}
d \left(a_1^2-4 d^2+x^2+y^2\right)-a_1 x y S=0,
\end{equation}

\begin{equation}\label{2of6}
d \left(a_2^2+a_3^2-4 d^2+x^2\right)-a_2 a_3 x S =0,
\end{equation}

\begin{equation}\label{3of6}
d \left(a_1^2+a_2^2-4 d^2+z^2\right)-a_1 a_2 z S =0,
\end{equation}

\begin{equation}\label{4of6}
d \left(a_3^2-4 d^2+y^2+z^2\right)-a_3 y z S=0.
\end{equation}

In addition, we still have the equations from the homogeneous polynomial coefficients, which include:

\begin{equation}\label{5of6}
d^4=c_4=\frac{1}{16} (-a_1 a_3 + a_2 y + x z)^2.
\end{equation}

Also, $a_2 x y z=4(c_4-c_1)=4 \left( \frac{d}{2S}c_3\right)$, so 

\begin{equation}\label{6of6}
a_2 x y z= \frac{d}{2S} \left( a_1 a_2 z+a_1 x y+a_2 a_3 x+a_3 y z\right).
\end{equation}

Subtracting \eqref{3of6} from \eqref{1of6}, we get

$$d(x^2+y^2-a_2^2-z^2)-a_1(xy-a_2z) S=0.$$
If $xy-a_2z=0$ then also $a_2^2+z^2=x^2+y^2$ and we conclude $a_2=x$ and $z=y$ or $a_2=y$ and $z=x$; other possibilities are prevented since $x,y$ and $z$ are positive. 
If $xy-a_2 z \ne 0$, then

\begin{equation}\label{a1}
a_1=\frac{d(x^2+y^2-a_2^2-z^2)}{(xy-a_2z)S}.
\end{equation}

Likewise, equations \eqref{2of6} and \eqref{4of6} show that if $a_2x-yz=0$ then $a_2=y$ and $z=x$ or $a_2=z$ and $y=x$; otherwise

\begin{equation}\label{a3}
a_3=\frac{d(a_2^2+x^2-y^2-z^2)}{(xa_2-yz)S}.
\end{equation}

In order to prove $A$ has the form stated in the theorem, we must show $a_2=y$ and $x=z$. Once that is established, the values for $\delta_1$ and $\delta_2$ are simply the roots of the quadratic in $a_1$ or $a_3$ given by any of \eqref{1of6} through \eqref{4of6}. Therefore we need only worry about the cases above where

$$(i) \, a_2=x \text{ and } z=y,  \text{ or } (ii) \, xy -a_2 z \ne 0,$$

and

$$  (iii) \, a_2=z \text{ and } y=x,  \text{ or } (iv) \,  x a_2-y z \ne 0. $$.

If (i) and (iii) hold, we obtain $a_2=y$ and $z=x$ as desired. If (i) and (iv) hold then 

$$a_3=\frac{2d(x^2-y^2)}{S(x^2-y^2)},$$
contradicting $|a_3| \leq 2d$. Assuming (ii) and (iii) hold contradicts $|a_1| \leq 2d$. So we need only consider the possibility that  $a_2x-yz \ne 0$ and $xy-a_2z \ne 0$. In that case, substitute expression \eqref{a1} for $a_1$ into equation \eqref{1of6}. After dividing by $d$, finding a common denominator and replacing the fraction with its numerator, we obtain the equation $s_1 d^2+t_1=0$ where

$$s_1=(x^2+y^2-a_2^2-z^2)^2-4(xy-a_2z)^2 S^2$$
and

$$t_1=(xy-a_2z)  S^2 \left( (a_2^2+z^2) xy-(x^2+y^2)a_2z \right)= S^2 (xy-a_2z)(xz-a_2y)(-a_2x+yz).$$

Repeating this procedure after substituting expression \eqref{a3} for $a_3$ into \eqref{2of6} yields $s_2 d^2+ t_2=0$ where

$$s_2=(x^2+a_2^2-y^2-z^2)^2-4(xa_2-yz)^2 S^2$$
and

$$t_2=(xa_2-yz)  S^2 \left( (y^2+z^2) xa_2-(x^2+a_2^2)yz \right)=S^2(xa_2-yz)(xz-a_2y)(a_2z-yx) .$$

Since $t_1=t_2$ and $d^2 \ne 0$, subtracting  $s_2 d^2+ t_2=0$ from $s_1 d^2+ t_1=0$ shows that $s_1=s_2$. Thus 

$$(x^2+y^2-a_2^2-z^2)^2-4(xy-a_2z)^2 S^2=(x^2+a_2^2-y^2-z^2)^2-4(xa_2-yz)^2 S^2.$$
After expanding and cancelling, this expression becomes

$$-4(a_2-y)(a_2+y)(x-z)(x+z)C^2=0,$$ which means $x=z$ or $y= \pm a_2$.

We will show that this leads to a contradiction, implying that our original assumption that $a_2 x=yz \ne 0$ and $xy-a_2z \ne 0$ is false. Therefore, the only remaining possibility is $x=z$ and $a_2=y$. 

First assume $a_2=-y$. Then \eqref{a1} and \eqref{a3} show that

$$a_1=\frac{d(x^2-z^2)}{y(x+z)S}=\frac{d(x-z)}{ yS}=-a_3.$$

So \eqref{6of6} shows

$$0>-y^2xz=\frac{da_1y(2x-2z)}{2S}=a_1y^2d\frac{(x-z)}{yS}=a_1^2y^2,$$
a contradiction.

Next assume $x=z$. Then  \eqref{a1} and \eqref{a3} show that

$$a_1=\frac{d(y^2-a_2^2)}{x(y-a_2)S}=\frac{d(y+a_2)}{ xS}=\frac{d(a_2^2-y^2)}{x(a_2-y)}=a_3.$$

So \eqref{6of6} shows

$$a_2yx^2=\frac{da_1x(y+a_2)}{S}=a_1x^2d\frac{(y+a_2)}{xS}=a_1^2x^2.$$
Thus $a_2y=a_1^2$, so \eqref{5of6} implies $x^4/16=d^4$. Therefore $z=x=2d$. 

Also, $x a_1S=d(y+a_2)$ shows that $2a_1S=y+a_2$. Squaring both sides of this equation yields 

$$4S^2 a_2y=y^2+2a_2y+a_2^2,$$
and since $a_2 y \geq 0$, we conclude that

$$4 a_2y>y^2+2a_2y+a_2^2,$$ 
also a contradiction. 

Finally, assume $y=a_2$ in which case \eqref{a1} and \eqref{a3} show 

$$a_1=\frac{d(x+z)}{yS}=a_3.$$
Now \eqref{6of6} shows $y^2xz=\frac{da_1y(x+z)}{2S}=y^2a_1^2.$ Therefore, $a_1^2=xz$, which combines with \eqref{5of6} to imply $y=2d$. Similarly to the previous case, the identity $2a_1dS=d(x+z)$ leads to the contradiction 

$$4xz> x^2+2xz+z^2.$$

In conclusion, the only possibility is $x=z$ and $y=a_2$. Now to determine the values of $a_1$ and $a_3$, substitute $x=z$ and $a_2=y$ back into  \eqref{1of6} through \eqref{4of6}. We obtain two identical quadratics in $a_1$ and $a_3$, so they are roots of the same quadratic in $t$, namely

\begin{equation}\label{quadratic}
 dt^2-x y St +d(x^2+y^2-4d^2)=0.
 \end{equation}

 In order to guarantee $a_1$ and $a_3$ are real, the discriminant of \eqref{quadratic} must be nonnegative. This means 
 
 $$x^2y^2 S^2 -4d^2(x^2+y^2-4d^2) \geq 0.$$
 Hence
 
 $$y^2(x^2 S^2 -4d^2) \geq 4d^2(x^2-4d^2),$$
 and inequality \eqref{2d} shows that 
 
 $$y^2 \leq \frac{ 4d^2(4d^2-x^2)}{4d^2-x^2 S^2}$$
 from which the inequality \eqref{yrange} follows. Note that the inequalities above also show that $x=2d$ is impossible, yielding $x \in (0, 2d)$.
 
 Now the possible values for $a_1$ and $a_3$ are the roots of \eqref{quadratic}, which are the values for $\delta_1$ and $\delta_2$ given in the theorem. In fact, $a_1$ and $a_3$ must be distinct roots of the quadratic if there are two distinct roots because if $a_1=a_3$, then \eqref{6of6} with $x=z$ and $y=a_2$ implies $a_1=a_3=\frac{xy S}{2d}$. This concludes the proof that the form of the matrix $A$ in the theorem is necessary.

 Sufficiency. Assume that $A$ has the form \eqref{Aform}. Without loss of generality, set $t=0$; since $A$ is real, its numerical range will be symmetric about the real axis. Any other line of symmetry will be obtained by rotating through the appropriate angle $t$. The NR generating polynomial associated with $A$ is 
 
 $$p_A(u,v,w)=\det \left[
\begin{array}{cccc}
 w & \frac{u x}{2}-\frac{i v x}{2} & \frac{u \delta _1}{2}-\frac{ i v \delta _1}{2} & \frac{u y}{2}-\frac{i v y}{2} \\
 \frac{u x}{2}+\frac{i v x}{2} & w & \frac{u y}{2}-\frac{i v y}{2} & \frac{u \delta _2}{2}-\frac{i v \delta _2}{2}  \\
 \frac{u \delta _1}{2}+\frac{ i v \delta _1}{2} & \frac{u y}{2}+\frac{i v y}{2} & w & \frac{u x}{2}-\frac{i v x}{2} \\
 \frac{u y}{2}+\frac{i v y}{2} & \frac{u \delta _2}{2}+\frac{ i v \delta _2}{2} & \frac{u x}{2}+\frac{i v x}{2} & w \\
\end{array}
\right],$$
which simplifies to

\begin{align*}p_A(u,v,w)&=w^4+ \left( d^4-\frac{x^2 y^2}{4}\right) u^4 + \left( \frac{x^2 y^2 S}{2 d}\right) u^3w+  \left(2d^4-\frac{ x^2 y^2}{4}\right) u^2v^2 \\
&+ \left( -2d^2 -\frac{x^2 y^2  S^2}{4 d^2}\right) u^2w^2+ \left( \frac{x^2 y^2 S}{2 d} \right) uv^2w+  \left( -2d^2 -\frac{x^2 y^2  S^2}{4 d^2}\right)v^2w^2 +d^4v^4.
\end{align*}

The first order partial derivatives of $p_A$ are:

\begin{align*} \frac{\partial p_A}{\partial u}(u,v,w)&=4u^3  \left( d^4-\frac{x^2 y^2}{4}\right) +2uv^2 \left(2 d^4-\frac{x^2 y^2}{4}\right)+ (3u^2w +v^2w)  \left( \frac{x^2 y^2 S}{2 d}\right) \\
&+2uw^2  \left( -2d^2 -\frac{x^2 y^2 S^2}{4 d^2}\right), \\
 \frac{\partial p_A}{\partial v}(u,v,w)&=4d^4v^3 +2u^2v  \left(2d^4-\frac{ x^2 y^2}{4}\right) +2vw^2  \left( -2d^2 -\frac{x^2 y^2  S^2}{4 d^2}\right) +2uvw \left( \frac{x^2 y^2 S}{2 d} \right), \\
 \end{align*}

and
\begin{align*}
\frac{\partial p_A}{\partial w}(u,v,w)&=4w^3 + \left( \frac{x^2 y^2 S}{2 d}\right) (u^3+ uv^2)+  \left( -2d^2 -\frac{x^2 y^2 S^2}{4 d^2}\right)( 2wv^2+2wu^2).
\end{align*}
 
It is straightforward to verify that if $(u_0,v_0)=\left( \frac{S}{d},  \frac{C}{d} \right)$ then 

\begin{equation*}
\left( \frac{\partial p_A}{\partial u}(u_0, \pm v,_01),\frac{\partial p_A}{\partial v}(u_0,\pm v_0,1),\frac{\partial p_A}{\partial w}(u_0,\pm v_0,1)\right)
=(0,0,0).
\end{equation*}
The discussion at the beginning of Section~3 now shows that if these singularities correspond to flat portions on the boundary of $W(A)$, then the distance from each line containing a flat portion to the origin is $1/\sqrt{u_0^2+v_0^2}=d$. If $(u_0, v_0)=\frac{1}{d}(S,C)$ then the slope of the corresponding line $\ell_{(u_0,v_0)}$ is $-\tan \left( \frac{\theta}{2} \right)$ and the slope of $\ell_{(u_0,-v_0)}$  is $\tan \left( \frac{\theta}{2} \right)$. These lines intersect on the negative real axis at the point $(-d \csc(\theta/2),0 )$ and the angle between them is $\theta$. Therefore, it only remains to show that the singularities $(u_0, \pm v_0)$ correspond to distinct double tangent lines and the matrix $A$ will have the numerical range asserted in the theorem.  Thus we will verify that the remaining conditions in Proposition \ref{flattest} hold.

 Since $A$ is real, the numerical range is symmetric about the real axis so it suffices to check the behavior of the point $(u_0,v_0,1)$. The second order partial derivatives of $p_A$ are:

\begin{align*}
\frac{\partial^2 p_A}{\partial^2 u} (u,v,w)&=12u^2 \left( d^4-\frac{x^2 y^2}{4}\right) +2v^2 \left(2 d^4-\frac{x^2 y^2}{4}\right)+6uw\left( \frac{x^2 y^2 S}{2 d}\right) +2w^2\left( -2d^2 -\frac{x^2 y^2S^2}{4 d^2}\right),\\
\frac{\partial^2 p_A}{\partial^2 v} (u,v,w)&=12d^4 v^2+2u^2 \left(2 d^4-\frac{x^2 y^2}{4}\right)+2uw\left( \frac{x^2 y^2 S}{2 d}\right) +2w^2\left( -2d^2 -\frac{x^2 y^2  S^2}{4 d^2}\right),\\
\frac{\partial^2 p_A}{\partial v \partial u} (u,v,w)&=4uv  \left(2 d^4-\frac{x^2 y^2}{4}\right) +2vw\left( \frac{x^2 y^2 S}{2 d}\right) .
 \end{align*}
 
When $(u,v)=(u_0,v_0)=\left( \frac{S}{d},  \frac{C}{d} \right)$  these values reduce to 
 \begin{equation}\label{414}\begin{aligned}
a_{11}=\frac{\partial^2 p_A}{\partial^2 u} (u_0,v_0,1)&=8d^2 S^2- \frac{ x^2y^2}{2d^2},\\
a_{22}=\frac{\partial^2 p_A}{\partial^2 v} (u_0,v_0,1)&=8d^2 C^2,\\
a_{12}=\frac{\partial^2 p_A}{\partial v \partial u} (u_0,v_0,1)&=8 d^2 S C.
\end{aligned}\end{equation}
Note that $a_{22}=8d^2C^2 \ne 0$ and $a_{12}^2-a_{11}a_{22}=4x^2y^2C^2>0$. In addition, 
$$ a_{11} u_0^2+2a_{12}u_0v_0+a_{22}v_0^2=8-\frac{S^2x^2y^2}{2d^4},$$
which is nonzero due to the bounds on $x$ and $y$ given in the theorem. Therefore, condition \eqref{distinct} in Proposition \ref{flattest} is satisfied. 
Next we will check condition \eqref{extremeline}.

$$p_A(u_0,v_0,\gamma)=\gamma^4+ \gamma^2 \left( -\frac{S^2x^2y^2}{4d^4}-2\right)+ \gamma \left( \frac{S^2x^2y^2}{2d^4}\right)+ (S^2+C^2)^2-\frac{S^2x^2y^2}{4d^4},$$
so
$$p_A(u_0,v_0,\gamma)=(\gamma^2-1)^2-(\gamma-1)^2\frac{S^2x^2y^2}{4d^4}=(\gamma-1)^2 \left( (\gamma+1)^2-\frac{S^2x^2y^2}{4d^4} \right).$$

Assume $p_A(u_0,v_0,\gamma)=0$ for $\gamma \ne 1$. Since $x^2 \leq 4d^2$ and $y^2 \leq 4 d^2$, it follows that $\frac{S^2x^2y^2}{4d^4}  \leq 4 S^2<4$ whence $(\gamma+1)<2$; thus \eqref{extremeline} holds and there is a flat portion on $\ell_{(u_0,v_0)}$. 

Finally, if we substitute the values of the second order partial derivatives and $(u_0,v_0)$ into the formula \eqref{lengthofflat} we obtain that the length of the flat portion is the value of $L$ in the statement of the theorem. 
\cvd 
 
 \section{Examples of nonparallel flat portions}\label{examples}
 \setcounter{equation}{0}
In this section, we will discuss two families of examples exhibiting every possible angle between flat portions.

The first family has a particularly simple form with one parameter. Fix $k \in (0, \sqrt{2})$. Define the 4-by-4 nilpotent matrix $A_k$ by

\begin{equation*}\label{Ak}
A_k =
\begin{bmatrix}
	0 & 1 & 0 & 1 \\
	0 & 0 & 1 & k \\
	0 & 0 & 0 & 1 \\
	0 & 0 & 0 & 0
\end{bmatrix}.
\end{equation*}

Note that $A_k$ satisfies the conditions of Theorem \ref{allangles} with $d=\frac{1}{\sqrt{2}}$, $x=y=1$, and $\theta=2 \arcsin(k/\sqrt{2})$. As $k$ ranges over the interval $(0, \sqrt{2})$, the angle $\theta$ covers all possible angles in the interval $(0, \pi)$. When $k=0$, the resulting numerical range has two horizontal parallel flat portions, so the angle between them is $\arcsin(0)$; when $k=\sqrt{2}$, the generating curve has a singularity that corresponds to a repeated tangent line instead of distinct tangent lines, so there is no flat portion on the boundary of the numerical range.

\begin{example} If $k=\sqrt{1 - \sqrt{3}/2}$ then  the angle $\theta$ between the flat portions is $\frac{\pi}{6}$.

\begin{center}
\includegraphics[scale=.35]{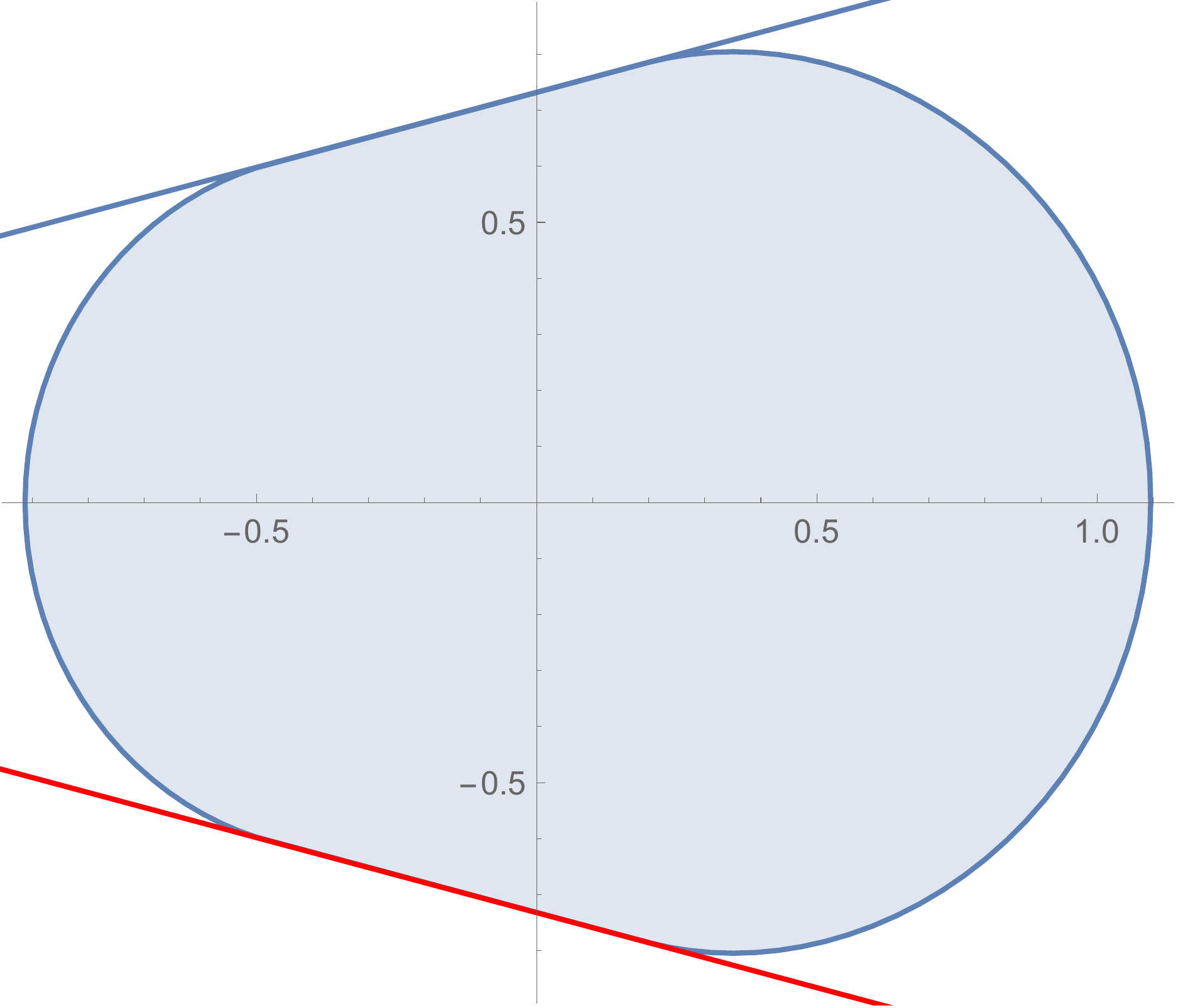}
\end{center}

\end{example}

\begin{example} If $k=\sqrt{\sqrt{\frac{\sqrt{5}}{8}+\frac{5}{8}}+1}$ then  the angle $\theta$ between the flat portions is $\frac{9 \pi}{10}$.

\begin{center}
\includegraphics[scale=.35]{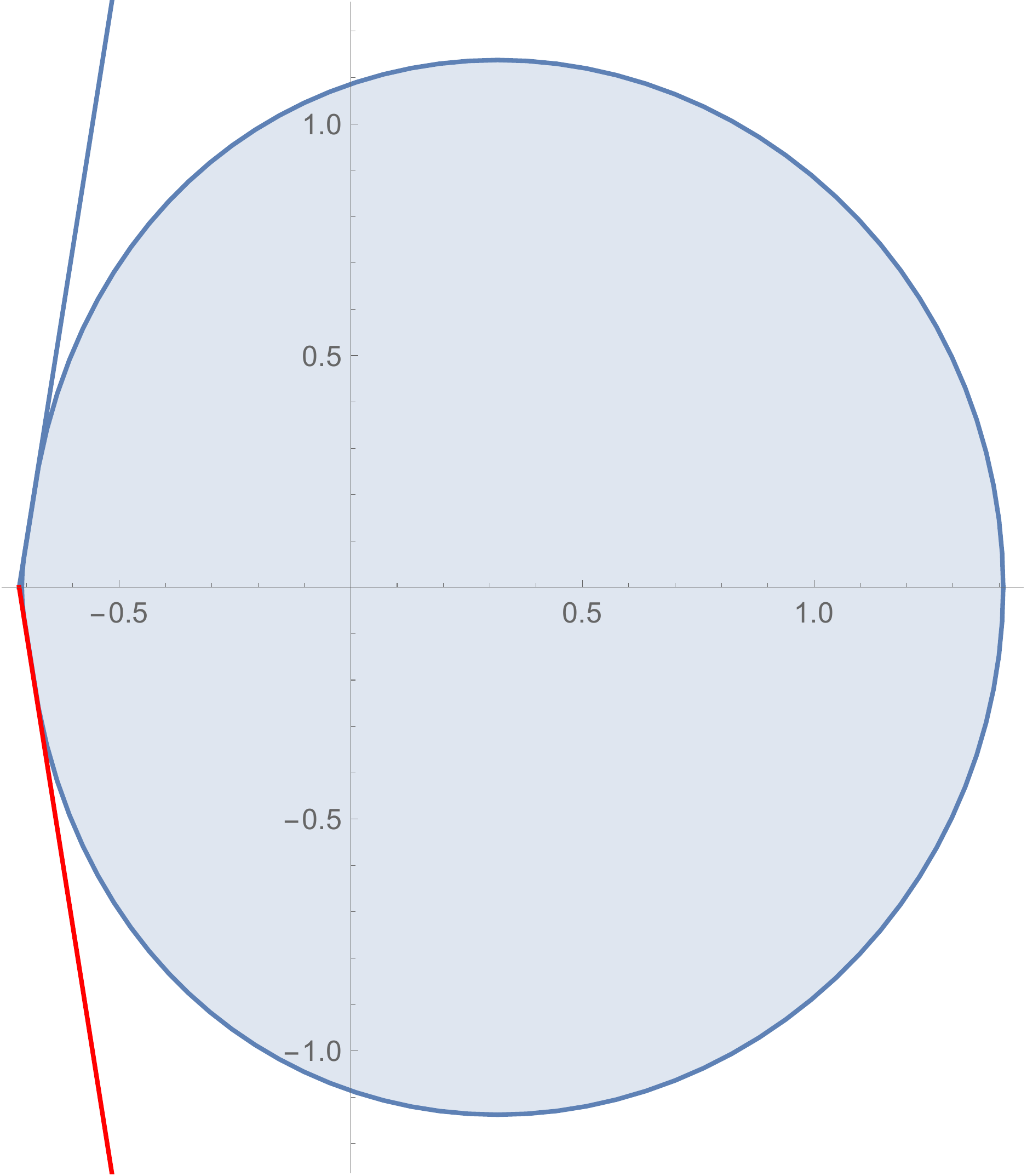}
\end{center}

\end{example}

Note that none of the $A_k$ matrices have numerical ranges where the flat portions attain the maximum length, because as Remark \ref{maxlength} describes, if the length is maximal and $d=\frac{1}{\sqrt{2}}$ then using the previous notation \eqref{SC}, $$x^2=y^2=4d^2/(1+C)=2/(1+C)>1.$$

We next define a two-parameter family of matrices with numerical ranges that exhibit flat portions of maximal length. Let $d>0$ and $\theta \in (0, \pi)$. 
Use the values of $x$ and $y$ in Remark \ref{maxlength} to show that for those values $\delta_1=\delta_2=\frac{2d S}{1+C}$, and define

\begin{equation*}\label{Mdtheta}
M_{d,\theta} =
\frac{2d}{\sqrt{1+C}}\begin{bmatrix}
	0 & 1 & \frac{S}{\sqrt{1+C}}& 1 \\
	0 & 0 & 1 &  \frac{S}{\sqrt{1+C}}\\
	0 & 0 & 0 & 1 \\
	0 & 0 & 0 & 0
\end{bmatrix}.
\end{equation*}

For each $d>0$ and $\theta \in (0,\pi)$, $W(M_{d,\theta})$ will have two flat portions of maximal length $L_{\mathrm{max}}=d$ on lines that are a distance $d$ from the origin and that meet at angle $\theta$. 

The values from \eqref{414} with $x^2y^2=\frac{16d^4}{(1+C)^2}$ can be substituted into \eqref{pointsoncurve} to find endpoints of the flat portion correspoding to the singularity $(S/d,C/d,1)$. The only change required for the values corresponding to $(S/d,-C/d,1)$ is $a_{12}=-8d^2SC$. 
\begin{example} The numerical range of $M_{1,\frac{2\pi}{3}}$ is plotted below. Note $L_{\mathrm{max}}=1$ is the length of each flat portion, and the purple dots denote the endpoints of these flat portions. 
\begin{center}
\includegraphics[scale=.35]{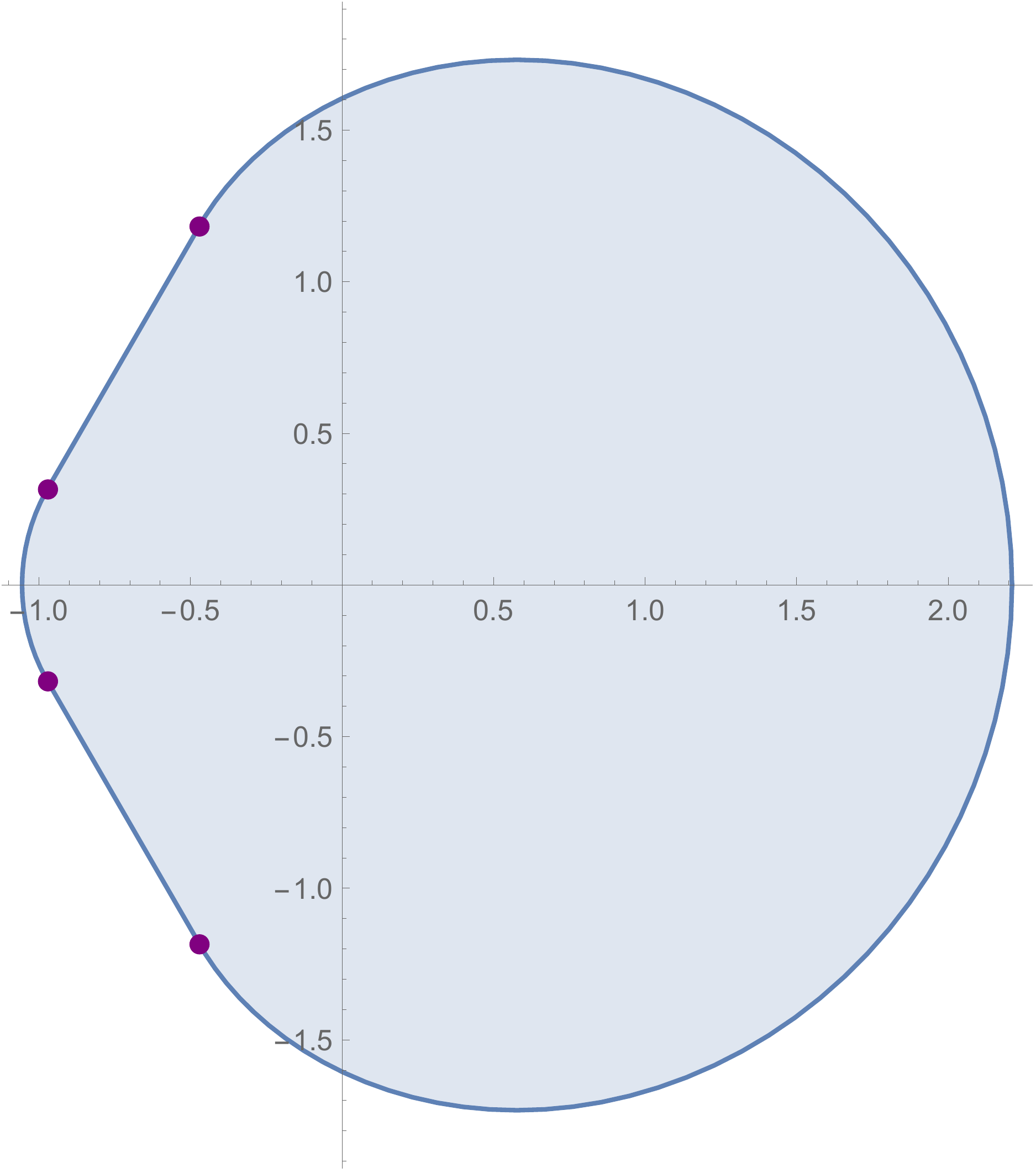}
\end{center}
\end{example}

\end{document}